\documentclass[11pt,reqno]{amsart} 
\usepackage{amsmath,amssymb,rawfonts}
\usepackage{tikz}
\usetikzlibrary{matrix,arrows}

\newtheorem{theorem}{Theorem}[section]
\newtheorem{lemma}[theorem]{Lemma}
\newtheorem{defi}[theorem]{Definition}
\newtheorem{prop}[theorem]{Proposition}
\newtheorem{coro}[theorem]{Corollary}

\begin{document} 
\title{INTERIOR OPERATORS AND TOPOLOGICAL CATEGORIES}
\author{Joaqu\'in Luna-Torres $^{\lowercase{a}}$\,\ and \,\ Carlos Orlando Ochoa C.$^{\lowercase{b}}$ }
\dedicatory{$^a$ Universidad Sergio Arboleda \\
$^b$  Universidad Distrital Francisco Jose de Caldas }
\email{$^a$oochoac@udistrital.edu.co; camicy@etb.net.co}
\email{$^b$ jlunator.gmail.com}
\begin{abstract}
The introduction of the categorical notion of closure operators  has unified various important notions and has led to interesting examples  and applications in diverse areas of mathematics (see for example, Dikranjan and Tholen (\cite{DT})). For a topological space it is well-known that the associated closure and interior operators provide equivalent descriptions of the topology, but this is not true in general. So, it makes sense to define and study the notion of interior operators $I$ in the context of a category $\mathfrak C$ and a fixed class $\mathcal M$ of monomorphisms in $\mathfrak C$ closed under composition in such a way that $\mathfrak C$ is finitely $\mathcal M$-complete and the inverse images of morphisms have both left and  right adjoint, which is the purpose of this paper.\\
Then we construct a concrete category  $\mathfrak C_{I}$ over $\mathfrak C$ which is a topological category. Furthermore, we provide some examples and discuss some of their properties: Kuratowski interior operator, Grothendieck interior operator, interior operators on Grothendieck topos and interior operators on the category of fuzzy topological spaces.
\end{abstract}
\maketitle 
\baselineskip=1.7\baselineskip
\section*{0. Introduction}
Kuratowski operators (closure, interior, exterior, boundary and others) have been used intensively in General Topology (\cite{K1}, \cite{K2}). Category Theory provides a variety of notions which expand on the lattice-theoretic concept of interior operator (\cite{Du}): For a topological space it is well-known that the associated closure and interior operators provide equivalent descriptions of the topology; but it is not generally true in other categories, consequently it makes sense to define and study the notion of interior operators $I$ in the context of a category $\mathfrak C$ and a fixed class $\mathcal M$ of monomorphisms in $\mathfrak C$ closed under composition in such a way that $\mathfrak C$ is finitely $\mathcal M$-complete and the inverse images of morphisms have both left and  right adjoint.

The paper is organized as follows: Following (\cite{DT}) we introduce, in section $1$, the basic categorial framework on subobjects, inveres images  and image factorization as needed throughout the paper. In section $2$, we present the concept of interior operator $I$ for suitable categories and then we construct a topological category $(\mathfrak C_I,U)$. Finally in section $3$ we provide some examples and discuss some of their properties: Kuratowski interior operator, Grothendieck interior operator, interior operators on Grothendieck topos and interior operators on the category of fuzzy topological spaces.

\section{Preliminaries on Subobjects, Inverse Images and its adjoints}
In this section we provide the basic categorial framework on subobjects, inveres images  and image factorization as needed throughout the paper.
\subsection{$\mathcal M$-subobjects}
In order to allow for sufficient flexibility, as in Dikrajan and Tholen \cite{DT}, we consider a category $\mathfrak C$ and a fixed class $\mathcal M$ of monomorphisms in $\mathfrak C$ which will play the role of subobjects. We assume that
\begin{itemize}
\item $\mathcal M$ is closed under composition with isomorphisms.
\item $\mathcal M$ contains al identity morphisms.
\end{itemize}
For every object $X$ in $\mathfrak C$, let $\mathcal M/X$ the class of all $\mathcal M$-morphisms with codomain $X$; the  relation given by 
\[
m\leqslant n \Leftrightarrow (\exists j)\,\  n\circ j=m
\]

\begin{equation}
\begin{tikzpicture}
\matrix(a)[matrix of math nodes,
row sep=3em, column sep=2.5em,
text height=1.5ex, text depth=0.25ex]
{M& &N\\
&X\\};
\path[->,font=\scriptsize](a-1-1) edge node[above]{$j$} (a-1-3);
\path[->,font=\scriptsize](a-1-1)edge node[below left]{$m$} (a-2-2);
\path[->,font=\scriptsize](a-1-3) edge node[below right]{$n$} (a-2-2);
\end{tikzpicture}
\end{equation}
is reflexive and transitive, hence $\mathcal M/X$ is a preordered class. Since $n$ is monic, the morphism $j$ is uniquely determined, and it is an isomorphism of $\mathfrak C$ if and only if $n\leqslant m$ holds; in this case $m$ and $n$ are isomorphic, and we write $m\cong n$. Of course, $\cong$ is an equivalence relation, and $\mathcal M/X$ modulo $\cong$ is a partially ordered class for which we can use all lattice-theoretic terminology. If $\widehat{m}$ denotes de $\cong$-equivalence class of $m$, we have, in particular, the equivalence
\[
m\cong n \Leftrightarrow \widehat{m} =\widehat{n}.
\]
From now on $\widehat{\mathcal M/X}$ denotes the partially orderd class $\mathcal M/X$ modulo $\cong$, and $m$ denotes the class $\widehat{m}$.
\subsection{Inverse images are $\mathcal M$-pullbacks}
For our fixed class $\mathcal M$ of mono\-morphisms in the category $\mathfrak C$, we say that $\mathfrak C$ has $\mathcal M$-pullbacks if, for every morphism $F:X\rightarrow Y$ and $n\in \widehat{\mathcal M/Y}$, a pullback diagram
\begin{equation}\label {pb}
\begin{tikzpicture}
\matrix(a)[matrix of math nodes,
row sep=3em, column sep=2.5em,
text height=1.5ex, text depth=0.25ex]
{M& N\\
X &Y\\};
\path[->,font=\scriptsize](a-1-1) edge node[above]{$f^{'}$} (a-1-2);
\path[->,font=\scriptsize](a-1-1)edge node[ left]{$m$} (a-2-1);
\path[->,font=\scriptsize](a-1-2) edge node[ right]{$n$} (a-2-2);
\path[->,font=\scriptsize](a-2-1) edge node[ above]{$f$} (a-2-2);
\end{tikzpicture}
\end{equation}
exists, with $m\in \widehat{\mathcal M/X}$. Of course, as an $\mathcal M$-subobject of $\mathfrak C$, $m$ is uniquely determined; it is called {\bf the inverse image of $n$ under $f$} and denoted by $f^{-1}(n):f^{-1}(N)\rightarrow X$. The pullback property of (\ref{pb}) yields that
\[
f^{-1} (-):\widehat{\mathcal M/Y}\rightarrow \widehat{\mathcal M/X}
\]
is an order-preserving map so that 
\[
k\leqslant n\Rightarrow f^{-1}(k)\leqslant f^{-1}(n)
\]
holds.
\begin{equation}
\begin{tikzpicture}
\matrix(a)[matrix of math nodes,
row sep=5em, column sep=0.8em,
text height=1.5ex, text depth=0.25ex]
{& f^{-1}(N)\ & &\ N&\\
f[{-1}(K)& & K& \\
X & & Y& \\};
\path[->,font=\scriptsize](a-2-1) edge node[left]{$f^{-1}(k)$} (a-3-1);
\path[->,font=\scriptsize](a-2-3)edge node[left]{$k$} (a-3-3);
\path[->,font=\scriptsize](a-1-4) edge node[right]{$n$} (a-3-3);
\path[->,font=\scriptsize](a-2-3) edge node[below]{$j$} (a-1-4);
\path[->,font=\scriptsize](a-1-2) edge  (a-1-4);
\path[->,font=\scriptsize](a-2-1) edge  (a-2-3);
\path[->,font=\scriptsize](a-2-1) edge  (a-1-2);
\path[->,font=\scriptsize](a-1-2) edge node[below right]{$f^{-1}(n)$} (a-3-1);
\path[->,font=\scriptsize](a-3-1) edge node[below]{$f$} (a-3-3);
\end{tikzpicture}
\end{equation}
\subsection{When the subobjects form a large-complete lattice}
If $\mathfrak C$ has $\mathcal M$-pullbacks and if $\mathcal M$ is closed under composition, the ordered class $\widehat{\mathcal M/X}$ has binary meets for every object $X$: one obtains the meet 
\[
m\land n:M\land N\rightarrow X
\]
as the diagonal of the pullback diagram
\begin{equation}\label {inf}
\begin{tikzpicture}
\matrix(a)[matrix of math nodes,
row sep=3em, column sep=2.5em,
text height=1.5ex, text depth=0.25ex]
{M\land N& N\\
M &X\\};
\path[->,font=\scriptsize](a-1-1) edge (a-1-2);
\path[->,font=\scriptsize](a-1-1)edge (a-2-1);
\path[->,font=\scriptsize](a-1-2) edge node[ right]{$n$} (a-2-2);
\path[->,font=\scriptsize](a-2-1) edge node[ above]{$m$} (a-2-2);
\end{tikzpicture}
\end{equation}
In general, for any $\mathcal M$, we say that $\mathfrak C$ has {\bf $\mathcal M$-intersections} if for every family $(m_i)_{\text{\tiny{$i\in I$}}}$ in $\widehat{\mathcal M/X}$ ($I$ may be a proper class or empty) a multiple pullback diagram
\begin{equation}
\begin{tikzpicture}
\matrix(a)[matrix of math nodes,
row sep=3em, column sep=2.5em,
text height=1.5ex, text depth=0.25ex]
{&M_i& \\
M & &X\\};
\path[->,font=\scriptsize](a-1-2) edge node[above right]{$m_i$} (a-2-3);
\path[->,font=\scriptsize](a-2-1) edge node[above]{$m$} (a-2-3);
\path[->,font=\scriptsize](a-1-2) edge node[above left]{$j_i$} (a-2-1);
\end{tikzpicture}
\end{equation}
exists in $\mathfrak C$ with $m\in \widehat{\mathcal M/X}$. One easily verifies that $m$ indeed assumes the role of the meet of $(m_i)_{\text{\tiny{$i\in I$}}}$ in $\widehat{\mathcal M/X}$. Hence we writes
\[
m=\bigwedge_{i\in I}m_i: \bigwedge_{i\in I}M_i\rightarrow X.
\]
\begin{prop}
If $\mathfrak C$ has  $\mathcal M$-intersections then every ordered class  $\widehat{\mathcal M/X}$ has the structure of a large-complete lattice, i.e. class-indexed meets and joins exist in $\widehat{\mathcal M/X}$, for every object $X\in \mathfrak C$.
\end{prop}
\begin{proof}
As usual, we construct the join of $(m_i)_{\text{\tiny{$i\in I$}}}$ in $\widehat{\mathcal M/X}$ as the meet of all upper bounds of $(m_i)_{\text{\tiny{$i\in I$}}}$ in $\widehat{\mathcal M/X}$.
\end{proof}
If $\mathfrak C$ has  also $\mathcal M$-pullbacks, it is easy to see that the join $m\in \widehat{\mathcal M/X}$ of $(m_i)_{\text{\tiny{$i\in I$}}}$ has the following categorical property: there are morphisms $j_i$, $ i\in I$, such that
\begin{enumerate}
\item $m\centerdot j_i=m_i$, for all $ i\in I$;
\item whenever we have commutative diagrams
\begin{equation}\label{union}
\begin{tikzpicture}
\matrix(a)[matrix of math nodes,
row sep=3em, column sep=3.5em,
text height=1.5ex, text depth=0.25ex]
{M_i&N \\
M& \\
X &Z\\};
\path[->,font=\scriptsize](a-1-1) edge node[above]{$u_i$} (a-1-2);
\path[->,font=\scriptsize](a-1-1) edge node[left]{$j_i$} (a-2-1);
\path[->,font=\scriptsize](a-2-1) edge node[left]{$m$} (a-3-1);
\path[->,font=\scriptsize](a-1-2) edge node[right]{$n$} (a-3-2);
\path[->,font=\scriptsize](a-3-1) edge node[above]{$v$} (a-3-2);
\end{tikzpicture}
\end{equation}
in $\mathfrak C$ with $m\in \mathcal M$, then there is a uniquely determined morphism $\omega: M\rightarrow N$ with $n\centerdot \omega=v\centerdot m$, and $\omega\centerdot j_i=u_i$, for all $ i\in I$.
\end{enumerate}
A subobject $m\in \widehat{\mathcal M/X}$ is called an {\bf $\mathcal M$-union} of $(m_i)_{\text{\tiny{$i\in I$}}}$ if this categorical property holds. 
Letting $v=1_X$ in (\ref{union}) we see that unions are joins in $\widehat{\mathcal M/X}$, hence we writes
\[
m=\bigvee_{i\in I}m_i: \bigvee_{i\in I}M_i\rightarrow X.
\]
When $I=\emptyset$, the union $\bigvee_{i\in I}m_i$ (if it exists) is called the {\bf trivial $\mathcal M$-subobject of $X$}; it is the least element of $\widehat{\mathcal M/X}$ and therefore denoted by $o_{\text{\tiny{$X$}}}: O_X\rightarrow X$.\\
Its characteristic categorical property (c.f. Diagram (\ref{union})) reads as follows: for every diagram
\begin{equation}\label {trivial}
\begin{tikzpicture}
\matrix(a)[matrix of math nodes,
row sep=3em, column sep=2.5em,
text height=1.5ex, text depth=0.25ex]
{O_X& N\\
X &Z\\};
\path[->,font=\scriptsize](a-1-1) edge node[left]{$o_{\text{\tiny{$X$}}}$} (a-2-1);
\path[->,font=\scriptsize](a-1-2) edge node[ right]{$n$} (a-2-2);
\path[->,font=\scriptsize](a-2-1) edge node[ above]{$v$} (a-2-2);
\end{tikzpicture}
\end{equation}
with $n\in \mathcal M$ there is a unequely determined morphism $\omega: O_X\rightarrow N$ with $n\centerdot \omega=v\centerdot o_{\text{\tiny{$X$}}}$.\\
Note that if the category $\mathfrak C$ has initial object $I$, then $o_{\text{\tiny{$X$}}}$ is the $\mathcal M$-part of the right $\mathcal M$-factorization of the only morphism $I\rightarrow X$. This is equivalent to the existence of  ``solution-set conditions''  (c. f. \cite{PJ}, I.4 or \cite{SM}, V.6)
\subsection{Review of pairs of adjoint maps}
Images of subobjects are given by left-adjoints of the maps $f^{-1} (-)$. We remember that a pair of mappings $\phi:P\rightarrow Q$ and $\psi:Q\rightarrow P$ between preordered classes $P, Q$ are adjoint if 
\begin{equation}\label{Gal}
\phi(m)\leqslant n \Leftrightarrow m\leqslant \psi(n)
\end{equation}
holds for all $m\in P$ and $n\in Q$,\,\ in which case one says that $\phi$ is left-adjoint of $\psi$ or $\psi$ is right-adjoint of $\phi$ and we writes $\phi \vdash\psi$. Note that adjoints determine each other uniquely, up to the equivalence relation given by $(x\cong y\Leftrightarrow x\leqslant y\,\ \text{and}\,\ y\leqslant x)$. In other words, in ordered classes adjoints determine each other uniquely.
\begin{lemma}\label{adj}
The following assertions are equivalent for any pair of mappings $\phi:P\rightarrow Q$ and $\psi:Q\rightarrow P$ between  large-complete lattices: 
\begin{enumerate}
\item[(i)] $\phi \vdash\psi$;
\item[(ii)] $\phi$ is order-preserving, and $\phi(m) = \bigwedge\{ n\in Q\mid m\leqslant \psi(n)\}$ holds for all $m\in P$;
\item[(iii)] $\psi$ is order-preserving, and $\psi(n) = \bigvee\{ m\in P\mid \phi(m)\leqslant n\}$ holds for all $n\in Q$;
\item[(iv)] $\phi$ and $\psi$ are order-preserving, and 
$$m\leqslant \psi(\phi(m))\,\,\ \text{and}\,\,\  \phi(\psi(n))\leqslant n$$
 holds for all $m\in P$ and $n\in Q$.
\end{enumerate}
\end{lemma}
\begin{proof}
\,\ $(i)\rightarrow (ii)\ \&\ (iii)$\,\,\,\,\ Putting $n=\phi(m)$ in (\ref{Gal}), we obtain\linebreak  $m\leqslant  \psi(\phi(m))$, hence $\phi(m)\in Q_m$,, where $Q_m =\{n\in Q\mid m\leqslant \psi(n)\}$, Furthermore, for all $n\in Q_m$ (\ref{Gal}) yields $\phi(m)\leqslant n$, hence $\phi(m)=\bigwedge Q_m$. This formula implies that $\phi$ is order-preserving. Dually we obtain the formula for $\psi$ as given in $(iii)$, and that $\psi$ is order-preserving.\\
\,\ $(ii)\rightarrow (iv)$\,\,\,\,\ As mentioned before, the given formula for $\phi$ implies its monotonicity. Furthermore, since $\phi(m)\in Q_m$, we have $m\leqslant \psi(\phi(m))$, and since $n\in Q_{\psi(n)}$, we have $\phi(\psi(n))\leqslant n$ for all $m\in P$ and $n\in Q$.\\
\,\ $(iii)\rightarrow (iv)$\,\,\,\,\ follows dually.\\
\,\ $(iv)\rightarrow (i)$\,\,\,\,\ $m\leqslant \psi(n)$ implies $\phi(m)\leqslant \phi(\psi(n))\leqslant n$, and $\phi(m)\leqslant n $ implies $m\leqslant  \psi(\phi(m))\leqslant \psi(n)$.
\end{proof}
The most important property of adjoints pairs is the preservation of joins and meets
\begin{prop}\label{preserve}
If $\phi \vdash\psi$, where $\phi:P\rightarrow Q$ and $\psi:Q\rightarrow P$  are mappings between  large-complete lattices, then $\phi$ preserves all joins and $\psi$ preserves all meets. Hence we have the formulas
\[
\phi \left( \bigvee_{i\in I}m_i\right)= \bigvee_{i\in I}\phi \left(m_i\right)\,\,\  \text{and}\,\,\ \psi\left(\bigwedge_{i\in I}n_i\right)=\bigwedge_{i\in I}\psi\left(n_i\right).
\]
Furthermore, $\phi\centerdot \psi\centerdot \phi=\phi$ and $\psi\centerdot \phi\centerdot \psi=\psi$, so that $\phi$ and $\psi$ give a biyective correspondence between $\phi(P)$ and $\psi(Q)$.
\end{prop}
\begin{proof}
By monotonicity of $\phi$,\,\ $\phi(m)$ is an upper bound of $\{\phi(m_i)\mid i\in I\}$, with $m=\bigvee_{i\in I}m_i$. For any other upper bound $n$, we have $m_i\leqslant \psi(n)$ for all $i\in I$ by (\ref{Gal}), hence $m\leqslant \psi(n)$. Application of (\ref{Gal}) again yields $\phi(m)\leqslant n$. This proves that $\phi$ preserves joins. The assertion for $\psi$ follows dually.\\
Furthermore, when applying the order-preserving map $\phi$ to the first inequality of $(iv)$ in the Lemma \ref{adj}, we obtain $\phi(m)\leqslant \phi(\psi(\phi(m))) $, and\linebreak  when exploting the second inequality in case $n=\phi(m)$, we obtain \linebreak $ \phi(\psi(\phi(m))) \leqslant \phi(m)$. Hence $\phi\centerdot \psi\centerdot \phi=\phi$ \ and\,\ $\psi\centerdot \phi\centerdot \psi=\psi$ follows dually.
\end{proof}
The converse of the first statement of Proposition \ref{preserve} holds as well
\begin{theorem}\label{r-adj}
Let $P, \ Q$ be partially ordered sets, then
\begin{enumerate}
\item If an order-preserving  map $\psi:Q\rightarrow P$ has left adjoint $\phi:P\rightarrow Q$, \  $\psi$ preserve all meets which exist in $Q$.
\item If $Q$ has all meets and $\psi$ preserves then, $\psi$ has a left adjoint.
\item If an order-preserving  map  $\phi:P\rightarrow Q$ has right adjoint $\psi:Q\rightarrow P$, \  $\phi$ preserve all joins which exist in $P$.
\item If $P$ has all joins and $\phi$ preserves then, $\phi$ has a right adjoint.
\end{enumerate}
\end{theorem}
\begin{proof}\

It suffices to show $(1)$ and $(2)$ since $(3)$ and $(4)$ follows by dualization. \\
$(1)$ Let $X$ a subset of $Q$ such that $\bigwedge X$ exists. Since $\psi$ is order-preserving, $\psi\left(\bigwedge X\right)$ is a lower bound of $\{\psi(x)\mid x\in X\}$. But if $p$ is any lower bound for this set, then we have $p\leqslant \psi(x)$ for all $x\in X$, whence $\phi(p)\leqslant x$ for all $x\in X$, so $\phi(p)\leqslant\bigwedge X$ and $p\leqslant \psi(\bigwedge X)$.\\
$(2)$ By definition of an adjoint, $\phi(p)$ most be the smallest $q\in Q$ satisfying $p\leqslant \psi(q)$. So consider $\phi(p)=\bigwedge\{q\in Q\mid p\leqslant \psi(q)\}$. Since $\psi$ preserve  meets, we have 
\[
p\leqslant \bigwedge\{\psi(q)\mid p\leqslant \psi(q)\}= \psi(\phi(p))\,\ \text{and}\,\ \phi(\psi(q))=\bigwedge\{y\mid  \psi(q)\leqslant\psi(y)\}\leqslant q
\]
since $q\in \{y\mid  \psi(q)\leqslant\psi(y)\}$. We can regard these inequalities as natural transformations $id_P \rightarrow \psi\centerdot\phi$ and $\phi\centerdot\psi\rightarrow id_Q$; so $\phi$ is left-adjoint of $\psi$.
\end{proof}
\subsection{Adjointness of image and inverse image}
Let $\mathfrak C$ have $\mathcal M$-pullbacks and for every $f:X\rightarrow Y$ in $\mathfrak C$, let $f^{-1} (-):\widehat{\mathcal M/Y}\rightarrow \widehat{\mathcal M/X}$ have a left adjoint 
\[
f(-):\widehat{\mathcal M/X}\rightarrow\widehat{\mathcal M/Y}.
\]
 For $m: M\rightarrow X$ in $\widehat{\mathcal M/X}$, we call $f(m): f(M)\rightarrow Y$ in $\widehat{\mathcal M/Y}$ the {\bf image} of $m$ under $f$; it is uniquely determined  by the property
 \begin{equation}\label{adjo}
 m\leqslant f^{-1} (n)\Leftrightarrow f(m)\leqslant n
 \end{equation}
 for all $n\in \widehat{\mathcal M/Y}$. Furthermore, (\ref{pb}) yields to the following formulas
 \begin{enumerate}
 \item $m\leqslant k\Rightarrow f(m)\leqslant (k)$;
 \item $m\leqslant f^{-1}(f(m))\,\,\ \text{and}\,\,\ \left( f^{-1}(n)\right)\leqslant n$ ;
 \item $f \left( \bigvee_{i\in I}m_i\right)= \bigvee_{i\in I}f \left(m_i\right)\,\,\  \text{and}\,\,\ f^{-1}\left(\bigwedge_{i\in I}n_i\right)=\bigwedge_{i\in I}f^{-1}\left(n_i\right)$.
 \end{enumerate}
 \begin{prop}
 When $\mathfrak C$ has $\mathcal M$-pullbacks and $(\mathcal E, \mathcal M)$-factorization system for morphisms, we have
 \begin{enumerate}
 \item If $f\in \mathcal M$, then  $f^{-1}(o_{\text{\tiny{$Y$}}})=o_{\text{\tiny{$X$}}}$ (provided the trivial subobject exists);
 \item $f\in \mathcal E$ if and only if f($o_{\text{\tiny{$X$}}})=o_{\text{\tiny{$Y$}}}$;
 \item If $f\in \mathcal M$, then $f^{-1}\left(f(m)\right)=m$ for all $m\in \widehat{\mathcal M/X}$;
 \item If $f\in \mathcal E$ and if $\mathcal E$ is stable under pullbacks, then $f\left(f^{-1}(n)\right)=n$ for all $n\in \widehat{\mathcal M/Y}$.
 \end{enumerate}
 \end{prop}
 Now, 
 \begin{prop}\label{rightadj}
 Let   $\mathfrak C$ be $\mathcal M$-complete and has $(\mathcal E, \mathcal M)$-factorization system for morphisms, and assume the existence of an object $P$ such that
 \[
 e\in \mathcal E\Leftrightarrow P\,\,\  \text{is projective with respect to}\,\,\  e 
 \]
 holds for every morphism $e$ in $\mathfrak C$. Then for a morphism $f:X\rightarrow Y$ and non-empty families $(m_i)_{\text{\tiny{$i\in I$}}}$ in $\widehat{\mathcal M/X}$ and $(n_i)_{\text{\tiny{$i\in I$}}}$ in $\widehat{\mathcal M/Y}$, we have:
 \begin{enumerate}
 \item If $f$ is a monomorphism, then $f\left(\bigwedge_{i\in I}m_i\right)=\bigwedge_{i\in I} f\left(m_i\right)$;
 \item If the sink $\left(j_i:N_i\rightarrow N  \right)_{i\in I}$ belonging to a union  $n= \bigvee_{i\in I}n_i$ as in (\ref{union}) has the property that for every $y:P\rightarrow N$ there is an $i\in I$ and a morphism $x:P\rightarrow N_i$ with $j_i\centerdot x=y$, then 
 $$f^{-1}\left( \bigvee_{i\in I}n_i\right)= \bigvee_{i\in I} f^{-1} \left(n_i\right)$$.
 \end{enumerate}
 \end{prop}
 \begin{proof}
 See \cite{DT}, p. 23 
 \end{proof}
 Observe that condition ($2$) of Proposition \ref{rightadj} and condition ($4$) of  Proposition \ref{r-adj} imply that $f^{-1} (-):\widehat{\mathcal M/Y}\rightarrow \widehat{\mathcal M/X}$ have a right adjoint 
\[
f_{*}(-):\widehat{\mathcal M/X}\rightarrow\widehat{\mathcal M/Y}.
\]
For $m: M\rightarrow X$ in $\widehat{\mathcal M/X}$; it is uniquely determined  by the property
 \begin{equation}
 m\leqslant f^{-1} (n)\Leftrightarrow f_{*}(m)\leqslant n
 \end{equation}
 for all $n\in \widehat{\mathcal M/Y}$. Furthermore, (\ref{adj}) implies that $f_{*}$ is an order-preserving map.

 \section{Interior Operators}
 Throughout this section, we consider a category  $\mathfrak C$  satisfying the conditions of Proposition \ref{rightadj}.
 \begin{defi}
 An interior operator $I$ of the category $\mathfrak C$ with respect to the class $\mathcal M$ of subobjects is given by a family $I =(i_{\text{\tiny{$X$}}})_{\text{$X\in \mathfrak C$}}$ of maps\linebreak $i_{\text{\tiny{$X$}}}:\widehat{\mathcal M/X}\rightarrow\widehat{\mathcal M/X}$ such that
 \begin{itemize}
 \item[($I_1)$] $\left(\text{Contraction}\right)$\,\  $i_{\text{\tiny{$X$}}}(m)\leqslant m$ for all $m\in \widehat{\mathcal M/X}$;
 \item[($I_2)$] $\left(\text{Monotonicity}\right)$\,\  If $m\leqslant k$ in $\widehat{\mathcal M/X}$, then $i_{\text{\tiny{$X$}}}(m)\leqslant i_{\text{\tiny{$X$}}}(k)$
 \item[($I_3)$] $\left(\text{Upper bound}\right)$\,\  $i_{\text{\tiny{$X$}}}(1_X)=1_X$.
 \end{itemize}
 \end{defi}
 \begin{defi}
 An $I$-space is a pair $(X,i_{\text{\tiny{$X$}}})$ where $X$ is an object of $\mathfrak C$ and $i_{\text{\tiny{$X$}}}$ is an interior operator on $X$.
 \end{defi}
 \begin{defi}
 A morphism $f:X\rightarrow Y$ of  $\mathfrak C$ is said to be $I$-continuous if 
 \begin{equation}\label{conti}
 f^{-1}\left( i_{\text{\tiny{$Y$}}}(m)\right)\leqslant i_{\text{\tiny{$X$}}}\left( f^{-1}(m)\right)
 \end{equation}
 for all $m\in \widehat{\mathcal M/Y}$.
 \end{defi}
 \begin{prop}
 Let $f:X\rightarrow Y$ and $g:Y\rightarrow Z$ be two morphisms  of  $\mathfrak C$ $I$-continuous then $g\centerdot f$ is a  morphism of  $\mathfrak C$  which is $I$-continuous.
 \end{prop}
 \begin{proof}
 Since $g:Y\rightarrow Z$ is $I$-continuous, we have  $g^{-1}\big( i_{\text{\tiny{$Z$}}}(m)\big)\leqslant i_{\text{\tiny{$Y$}}}\big( g^{-1}(m)\big)$
  for all $m\in \widehat{\mathcal M/Z}$, it fallows that $$f^{-1}\Big(g^{-1}\big( i_{\text{\tiny{$Z$}}}(m)\big)\Big)\leqslant f^{-1}\Big( i_{\text{\tiny{$Y$}}}\big( g^{-1}(m)\big)\Big);$$
  now,  by the $I$-continuity of $f$, $$f^{-1}\Big( i_{\text{\tiny{$Y$}}}\big( g^{-1}(m)\big)\Big)\leqslant i_{\text{\tiny{$X$}}}\Big( f^{-1}\big(g^{-1}(m)\big)\Big),$$ 
  threfore $$f^{-1}\Big(g^{-1}\big( i_{\text{\tiny{$Z$}}}(m)\big)\Big)\leqslant i_{\text{\tiny{$X$}}}\Big( f^{-1}\big(g^{-1}(m)\big)\Big),$$
   that is to say 
   $$(g\centerdot f)^{-1}\big( i_{\text{\tiny{$Z$}}}(m)\big)\Big)\leqslant i_{\text{\tiny{$X$}}}\Big( (g\centerdot f)^{-1}(m)\Big)$$
 for all $m\in \widehat{\mathcal M/Z}$. This complete the proof.   
 \end{proof}
  As a consequence we obtain
 \begin{defi}
 The category $\mathfrak C_{\text{\tiny{$I$}}}$ of $I$-spaces comprises the following data:
 \begin{enumerate}
 \item {\bf Objects}: Pairs $(X,i_{\text{\tiny{$X$}}})$ where $X$ is an object of $\mathfrak C$ and $i_{\text{\tiny{$X$}}}$ is an interior operator on $X$.
\item {\bf Morphisms}: Morphisms of $\mathfrak C$ which are $I$-continuous.
 \end{enumerate}
 \end{defi}
 \subsection{The lattice structure of all interior operators}
 For a category $\mathfrak C$  satisfying the conditions of Proposition \ref{rightadj} we consider the conglomerate
 \[
 Int(\mathfrak C,\mathcal M)
 \]
 of all  interior operators on $\mathfrak C$ with respect to $\mathcal M$. It is ordered by
 \[
 I\leqslant J \Leftrightarrow i_{\text{\tiny{$X$}}}(n)\leqslant j_{\text{\tiny{$X$}}}(n), \,\,\ \text{for all $n\in \widehat{\mathcal M/X}$ and all $X$  object of $\mathfrak C$. }
 \]
 This way $Int(\mathfrak C, \mathcal M)$ inherits a lattice structure from $\mathcal M$:
 \begin{prop}
 For $\mathfrak C$\,\ $\mathcal M$-complete, every family $(I_{\text{\tiny{$\lambda$}}})_{\text{\tiny{$\lambda\in \Lambda$}}}$ in $Int(\mathfrak C, \mathcal M)$ has a join $\bigvee\limits_{\text{\tiny{$\lambda\in \Lambda $}}}I_{\text{\tiny{$\lambda $}}}$ and a meet $\bigwedge\limits_{\text{\tiny{$\lambda\in \Lambda $}}}I_{\text{\tiny{$\lambda $}}}$ in $Int(\mathfrak C, \mathcal M)$. The discrete interior operator
 \[ I_{\text{\tiny{$D$}}}=({i_{\text{\tiny{$D$}}}}_{\text{\tiny{$X$}}})_{\text{$X\in \mathfrak C$}}\,\,\ \text{with}\,\,\ {i_{\text{\tiny{$D$}}}}_{\text{\tiny{$X$}}}(m)=m\,\,\ \text{for all}\,\ m\in \widehat{\mathcal M/X}
 \]
  is the largest element in $Int(\mathfrak C, \mathcal M)$, and the trivial interior operator
  \[ 
I_{\text{\tiny{$T$}}}=({i_{\text{\tiny{$T$}}}}_\text{\tiny{$X$}})_{\text{$X\in \mathfrak C$}}\,\,\ \text{with}\,\,\ {i_{\text{\tiny{$T$}}}}_{\text{\tiny{$X$}}}(m)=
  \begin{cases}
o_ {\text{\tiny{$X$}}}& \text{for all}\,\ m\in \widehat{\mathcal M/X},\,\ m\ne 1_{\text{\tiny{$X$}}}\\
1_{\text{\tiny{$X$}}} &\text {if}\,\ m=1_{\text{\tiny{$X$}}}
\end{cases}
 \]
  
  is the least one.
 \end{prop}
 \begin{proof}
 For $\Lambda\ne\emptyset$, let $\widetilde{I}=\bigvee\limits_{\text{\tiny{$\lambda\in \Lambda $}}}I_{\text{\tiny{$\lambda $}}}$, then 
 \[
 \widetilde{i_{\text{\tiny{$X$}}}}=\bigvee\limits_{\text{\tiny{$\lambda\in \Lambda $}}} {i_{\text{\tiny{$\lambda $}}}}_{\text{\tiny{$X$}}},
 \]
 for all  $X$ object of $\mathfrak C$, satisfies
 \begin{itemize}
 \item $ \widetilde{i_{\text{\tiny{$X$}}}}(m)\leqslant m$,  because ${i_{\text{\tiny{$\lambda $}}}}_{\text{\tiny{$X$}}}(m)\leqslant m$ for all $m\in \widehat{\mathcal M/X}$ and for all $\lambda \in \Lambda$.
 \item If $m\leqslant k$ in $\widehat{\mathcal M/X}$ then ${i_{\text{\tiny{$\lambda $}}}}_{\text{\tiny{$X$}}}(m)\leqslant {i_{\text{\tiny{$\lambda $}}}}_{\text{\tiny{$X$}}}(k)$  for all $m\in \widehat{\mathcal M/X}$ and for all $\lambda \in \Lambda$, threfore $ \widetilde{i_{\text{\tiny{$X$}}}}(m)\leqslant \widetilde{i_{\text{\tiny{$X$}}}}(k)$.
 \item Since ${i_{\text{\tiny{$\lambda $}}}}_{\text{\tiny{$X$}}}(1_{{\text{\tiny{$X$}}}})=1_{{\text{\tiny{$X$}}}} $ for all $m\in \widehat{\mathcal M/X}$ and for all $\lambda \in \Lambda$, we have that  $ \widetilde{i_{\text{\tiny{$X$}}}}(1_{{\text{\tiny{$X$}}}})=1_{{\text{\tiny{$X$}}}}$.
 \end{itemize}
 Similarly  $\bigwedge\limits_{\text{\tiny{$\lambda\in \Lambda $}}}I_{\text{\tiny{$\lambda $}}}$,\,\  $ I_{\text{\tiny{$D$}}}$ and $I_{\text{\tiny{$T$}}}$ are interior operators.
 \end{proof}
 \begin{coro}\label{complete}
  For $\mathfrak C$\,\ $\mathcal M$-complete and for every  object $X$ of $\mathfrak C$
  \[
  Int(X) = \{i_{\text{\tiny{$X$}}}\mid i_{\text{\tiny{$X$}}}\,\ \text{ is an interior operator on}\,\ X\}
  \]
  is a complete lattice.
 \end{coro}

 \subsection{Initial interior operators}
Let  $\mathfrak C$  be a category satisfying the conditions of Proposition \ref{rightadj}, let  $(Y,i_{\text{\tiny{$Y$}}})$ be an object of  $\mathfrak C_{\text{\tiny{$I$}}}$ and let $X$ be an object of $\mathfrak C$. For each morphism  $f:X\rightarrow Y$ in $\mathfrak C$ we define on $X$ the operotor 
\begin{equation} \label{initial}
i_{\text{\tiny{$X_{f}$}}}:=f^{-1}\centerdot i_{\text{\tiny{$Y$}}}\centerdot f_{*}.
\end{equation}
\begin{prop}\label{ini-cont}
The operator (\ref{initial}) ia an interior operator on $X$ for which the morphism $f$ is $I$-continuous.
\end{prop}
\begin{proof}\
\begin{enumerate}
\item[($I_1)$] $\left(\text{Contraction}\right)$\,\  $i_{\text{\tiny{$X_{f}$}}}(m)= f^{-1}\centerdot i_{\text{\tiny{$Y$}}}\centerdot f_{*}(m)\leqslant f^{-1}\centerdot  f_{*}(m)\leqslant m$ for all $m\in \widehat{\mathcal M/X}$;
 \item[($I_2)$] $\left(\text{Monotonicity}\right)$\,\   $m\leqslant k$ in $\widehat{\mathcal M/X}$, implies $f_{*}(m)\leqslant f_{*}(k)$, then\linebreak $i_{\text{\tiny{$Y$}}}\centerdot f_{*}(m)\leqslant i_{\text{\tiny{$Y$}}}\centerdot f_{*}(k)$, consequently  $ f^{-1}\centerdot i_{\text{\tiny{$Y$}}}\centerdot f_{*}(m)\leqslant f^{-1}\centerdot i_{\text{\tiny{$Y$}}}\centerdot f_{*}(k)$;
 \item[($I_3)$] $\left(\text{Upper bound}\right)$\,\  $i_{\text{\tiny{$X_{f}$}}}(1_X)=f^{-1}\centerdot i_{\text{\tiny{$Y$}}}\centerdot f_{*}(1_X)=1_{X}$.
\end{enumerate}
Finally, 
\begin{align*}
f^{-1}\big(i_{\text{\tiny{$Y$}}}(n)\big)&\leqslant f^{-1}\big(i_{\text{\tiny{$Y$}}}\centerdot f_{*}\centerdot f^{-1}(n)\big)=(f^{-1}\centerdot i_{\text{\tiny{$Y$}}}\centerdot f_{*})\big(f^{-1}(n)\big)\\
&= i_{\text{\tiny{$X_{f}$}}}\big(f^{-1}(n)\big),
\end{align*}
 for all $n\in \widehat{\mathcal M/Y}$.
\end{proof}
It is clear that $ i_{\text{\tiny{$X_{f}$}}}$ is the coarsest interior operator on $X$ for which the morphism $f$ is $I$-continuous; more precisaly
\begin{prop}\label{unique}
Let $(Z,i_{\text{\tiny{$Z$}}})$ and $(Y,i_{\text{\tiny{$Y$}}})$ be objects of $\mathfrak C_{\text{\tiny{$I$}}}$, and let $X$ be an object of  $\mathfrak C$. For each morphism  $g:Z\rightarrow X$ in  $\mathfrak C$ and for \linebreak $f:(X,i_{\text{\tiny{$X_{f}$}}})\rightarrow (Y,i_{\text{\tiny{$Y$}}})$ an $I$-continuous morphism, $g$  is $I$-continuous if and only if $g\centerdot f$ is $I$-continuous.
\end{prop}
\begin{proof}
Suppose that $g\centerdot f$ is $I$-continuous, i. e.
$$(f\centerdot g)^{-1}\big(i_{\text{\tiny{$Y$}}}(n)\big)\leqslant i_{\text{\tiny{$Z$}}}\big( (f\centerdot g)^{-1}(n) \big)$$
 for all $n\in \widehat{\mathcal M/Y}$. Then, for all $m\in \widehat{\mathcal M/X}$, we have
\begin{align*}
 g^{-1}\big(i_{\text{\tiny{$X_{f}$}}}(m)\big)&=g^{-1}\big(f^{-1}\centerdot i_{\text{\tiny{$Y$}}}\centerdot f_{*}(m)\big)=(f\centerdot g)^{-1}\big( i_{\text{\tiny{$Y$}}}( f_{*}(m)) \big)\\
 &\leqslant i_{\text{\tiny{$Z$}}}\big( (f\centerdot g)^{-1}(f_{*}(m) ) \big)=i_{\text{\tiny{$Z$}}}\big( g^{-1}\centerdot f^{-1}\centerdot  f_{*} (m)\big)\\
 &\leqslant i_{\text{\tiny{$Z$}}}\big( g^{-1}(m)\big).\\
\end{align*}
\end{proof}
As a consequence of corollary(\ref{complete}), proposition(\ref{ini-cont}) and proposition (\ref{unique}) (cf. \cite{AHS} or \cite{JM}), we obtain 
 \begin{theorem}
 Let $\mathfrak C$ be an $\mathcal M$-complete category then the concrete category $(\mathfrak C_{\text{\tiny{$I$}}}, U)$ over $\mathfrak C$ is a topological category.
 \end{theorem}
 \subsection{Open subobjects}
 \begin{defi}
 An $\mathcal M$-subobject $m: M\rightarrow X$  is called $I$-open (in $X$) if it is isomorphic to its $I$-interior, that is: if $j_m: i_{\text{\tiny{$X$}}}(M) \rightarrow M$ is an isomorphism.
 \end{defi}
 The $I$-continuity condition (\ref{conti}) implies that $I$-openness is preserve by inverse images:
 \begin{prop}
 Let $f:X\rightarrow Y$ be a morphism in $\mathfrak C$. If $n$ is $I$-open in $Y$, then $f^{-1}(n)$ is $I$-open in $X$.
 \end{prop}
 \begin{proof}
 If $n\cong i_{\text{\tiny{$Y$}}}(n)$ then $f^{-1}(n)=f^{-1}\big(i_{\text{\tiny{$Y$}}}(n)\big)\leqslant i_{\text{\tiny{$X$}}}\big(f^{-1}(n)\big)$, \linebreak so $i_{\text{\tiny{$X$}}}\big(f^{-1}(n)\big)\cong f^{-1}(n)$.
 \end{proof}
Let $\mathcal M^{I}$ be the class of $I$-open $\mathcal M$-subobjects. The last proposition asserts that $\mathcal M^{I}$ is stable under pullback, therefore
 \begin{coro}
 If, for monomorphisms $m$ and $n$,\,\ $n\centerdot m$ is an $I$-open\linebreak $\mathcal M$-subobject, then $m$ is an $I$-open $\mathcal M$-subobject.
 \end{coro}
 
 \section{Examples of Interior Operators}

 \subsection{Kuratowski interior operator}
The Kuratowski interior operator $I=(i_{\text{\tiny{$X$}}})_{X\in Sets}$ is described as follows (cf. \cite{Du}):
\begin{defi}
Let $X$ be a set and $ i_{\text{\tiny{$X$}}}:\wp(X)\rightarrow \wp(X)$ a map such that:
\begin{enumerate}
\item $i_{\text{\tiny{$X$}}}(X)=X$.
\item $i_{\text{\tiny{$X$}}}(A)\subseteq A$\qquad for all $A\in \wp(X)$.
\item  $i_{\text{\tiny{$X$}}}\centerdot i_{\text{\tiny{$X$}}}(A)=  i_{\text{\tiny{$X$}}}(A)$ \qquad for all $A\in \wp(X)$.
\item $i_{\text{\tiny{$X$}}}(A\cap B)=  i_{\text{\tiny{$X$}}}(A)\cap i_{\text{\tiny{$X$}}}(B)$\qquad for all $A,B\in \wp(X)$.
\end{enumerate}
\end{defi}
Then $\tau =\{A\in \wp(X)\mid i_{\text{\tiny{$X$}}}(A)=A\}$ is a topology on $X$.
 \subsection{Grothendieck interior operator}
 Let $ \mathfrak C$ be a small category, and let $\mathbf{Sets^{\mathfrak C^{op}}}$ be the corresponding functor category (cf. \cite{MM}). As usual, we write
 \[
 y:\mathfrak C\rightarrow  \mathbf{Sets^{\mathfrak C^{op}}}
 \]
 for the Yoneda embedding: $y(C)= Hom_{\mathfrak C}(-,C)$. Recall that 
 \begin{enumerate}
 \item A sieve $S$ on $C$ is a subobject $S\subseteq y(C)$ in $\mathbf{Sets^{\mathfrak C^{op}}}$. We write $Sub\ y(C)$ for the class of subobjects of $y(C)$.
 \item A sieve $S$ on $C$ is a right ideal of morphisms in $ \mathfrak C$, all with codomain $C$.
 \item If $S$ is a sieve on $C$ and $h: D\rightarrow C$ is any arrow to $C$, then $$h^{*}(S)=\{g\mid cod(g)=D\,\ h\centerdot g\in S\}$$ is a sieve on $D$.
 \item $t_{C}=\{f\mid cod(f)=C\}$ is the maximal sieve on $C$
 \end{enumerate}
 \begin{defi}
 An interior operator $I$ of the category $\mathfrak C$  is given by a family $I =(i_{\text{\tiny{$y(C)$}}})_{\text{$C\in \mathfrak C$}}$ of maps\,\ $i_{\text{\tiny{$y(C)$}}}:Sub\ y(C)\rightarrow Sub\ y(C)$ such that
 \begin{itemize}
 \item[($I_1)$] $\left(\text{Contraction}\right)$\,\  $i_{\text{\tiny{$y(C)$}}}(S)\leqslant S$ for all $S\in Sub\ y(C)$;
 \item[($I_2)$] $\left(\text{Monotonicity}\right)$\,\  If $S_1\leqslant S_2$ in $Sub\ y(C)$, then $i_{\text{\tiny{$y(C)$}}}(S_1)\leqslant i_{\text{\tiny{$y(C)$}}}(S_2)$
 \item[($I_3)$] $\left(\text{Upper bound}\right)$\,\  $i_{\text{\tiny{$y(C)$}}}(t_C)=t_C$.
 \end{itemize}
 \end{defi}
\begin{prop}
Suppose that  $\mathfrak C$ have $(\mathcal E, \mathcal M)$-factorization with $\mathcal M$-pullbacks, and $\mathcal E$ is stable under pullbacks. Then 
the function $J$ which assigns to each object $C$ of $\mathfrak C$ the collection $J(C)=\{S\mid S\,\ \text{is $I$-open}\}$ is a Grothendieck topology on $\mathfrak C$, whenever there exists an $\mathcal E$-morphisms  in each sieve $S$.
\end{prop}
\begin{proof}\
\begin{enumerate}
\item Clearly, $t_C \in J(C)$.
\item Suppose that $S\in J(C)$ and $h:D\rightarrow C$ is any arrow to $C$. Then for 
\[
i_{\text{\tiny{$y(D)_{h}$}}}= h^{*}\centerdot i_{\text{\tiny{$y(C)$}}}\centerdot h_{*},
\]
we have 
\[
i_{\text{\tiny{$y(D)_{h}$}}}\big( h^{*}(S) \big) =h^{*}\centerdot i_{\text{\tiny{$y(C)$}}}\centerdot h_{*}\big( h^{*}(S) \big)
\geqslant h^{*}\centerdot i_{\text{\tiny{$y(C)$}}}\centerdot (S) =h^{*}(S),
\]
consequently, $h^{*}(S)\in J(D)$.
\item Let $S$ be in $J(C)$, and let $R$ be any sieve on $C$ such that $h^{*}(R)\in J(D)$ for all $h: D\rightarrow C$ in $S$. Since there exists an $\mathcal E$-morphisms $g$  in $S$, and since $g_{*}\centerdot g^{*}(R)\cong R$, it follows that 
\[
R\cong g_{*} \left( g^{*}(R)\right)\cong  g_{*}\left(i_{\text{\tiny{$y(D)$}}}\left( g^{*}(R)\right)\right)=g_{*}\left(g^{*}\centerdot i_{\text{\tiny{$y(C)$}}}\centerdot g_{*}\left( g^{*}(R)\right)\right)\cong i_{\text{\tiny{$y(C)$}}}(R).
\]
\end{enumerate}
\end{proof} 
 \subsection{Interior operators on Grothendieck topos}
 Recall that a Grothen\-dieck topos is a category which is equivalent to the category $Sh(\mathfrak C,J)$ of sheaves on some site $(\mathfrak C,J)$ (cf. \cite{MM}). Furthermore, for any sheaf $E$ on a site $(\mathfrak C,J)$, the lattice $Sub\ (E)$ of all subsheaves of $E$ is a complete Heyting algebra. It is also true that any morphism $\phi:E\rightarrow F$ of sheaves on a site induces a functor on the corresponding partially ordered sets of subsheaves,
 \begin{equation}
 \phi^{-1}: Sub(F)\rightarrow Sub(E)
 \end{equation}
 by pullback. Moreover, this functor has both a left and a right adjoint:
 \begin{equation}
 \exists_{\phi} \dashv \phi^{-1}\dashv \forall_{\phi}.
 \end{equation}
 \begin{defi}
 An interior operator $I$ of the category $Sh(\mathfrak C,J)$ of sheaves on some site $(\mathfrak C,J)$  is given by a family $I =(i_{\text{\tiny{$E$}}})_{\text{$E\in Sh(\mathfrak C,J)$}}$ of maps\linebreak $i_{\text{\tiny{$E$}}}:Sub\ E\rightarrow Sub\ E$ such that
 \begin{itemize}
 \item[($I_1)$] $\left(\text{Contraction}\right)$\,\  $i_{\text{\tiny{$E$}}}(A)\leqslant A$ for all $A\in Sub\ E$;
 \item[($I_2)$] $\left(\text{Monotonicity}\right)$\,\  If $A\leqslant B$ in $Sub\ E$, then $i_{\text{\tiny{$E$}}}(A)\leqslant i_{\text{\tiny{$E$}}}(B)$
 \item[($I_3)$] $\left(\text{Upper bound}\right)$\,\  $i_{\text{\tiny{$E$}}}(E)=E$.
 \end{itemize}
 \end{defi}
 As a consequence we have a category $Sh(\mathfrak C,J){I}$ whose objects are pairs $(E, i_{\text{\tiny{$E$}}})$ where $E$ is a sheave on the site $(\mathfrak C,J)$, and whose morphisms are morphisms of $Sh(\mathfrak C,J)$ which are $I$-continuous; i. e.,  morphisms $\phi:E\rightarrow F$ such that
  $$\phi^{-1}\left( i_{\text{\tiny{$F$}}}(B)\right)\leqslant i_{\text{\tiny{$E$}}}\left( \phi^{-1}(B)\right)$$ 
  for all $B\in Sub\ E$.\\
  Given an object $(F, i_{\text{\tiny{$F$}}})$ of $Sh(\mathfrak C,J){I}$ and a morphism $\phi:E\rightarrow F$ of $Sh(\mathfrak C,J){I}$, defining $$i_{\text{\tiny{$E_{\phi}$}}}:=\phi^{-1}\centerdot i_{\text{\tiny{$F$}}}\centerdot \forall_{\phi},$$ it is clear that 
\begin{prop}
 The concrete category $(Sh(\mathfrak C,J)_{I}, U)$ over $Sh(\mathfrak C,J)$ is a topological category.
 \end{prop}
 \subsection{Interior operators on the category $LF$-Top of fuzzy topological spaces}
  Given a $GL$-monoid $(L,\leqslant, \otimes)$ (for example, a complete Heyting algebra or a complete $MV$- algebra), for any set $X$ (cf:\cite{HS}),
 \begin{defi}
 A mapping $\mathcal{I}: L^{X}\times L\to L^{X}$ is called an $L$-fuzzy interior operator on $X$ if and only if $\mathcal{I}$ satisfies the following conditions:
 \begin{itemize}
 \item[($I_1)$] $\mathcal{I}(1_X,\alpha)= 1_X$, \,\ for all $\alpha\in L$.
 \item[($I_2)$] $\mathcal{I}(g,\beta)\leqslant \mathcal{I}(f,\alpha)$ whenever $g\leqslant f$ and $\alpha\leqslant \beta$.
 \item[($I_3)$] $\mathcal{I}(f,\alpha)\otimes \mathcal{I}(g,\beta)\leqslant \mathcal{I}(f\otimes g,\alpha\otimes\beta)$.
 \item[($I_4)$] $\mathcal{I}(f,\alpha)\leqslant f$.
 \item[($I_5)$] $\mathcal{I}(f,\alpha)\leqslant \mathcal{I}\big(\mathcal{I}(f,\alpha)\big)$.
 \item[($I_6)$] $\mathcal{I}(f,\bot)=f$.
 \item[($I_7)$] If $\emptyset\ne K\subseteq L$ and $\mathcal{I}(f,\alpha)=f^{0}$, then $\mathcal{I}(f,\bigvee K)=f^{0}$.
 \end{itemize}
 \end{defi}
 Given an $L$-fuzzy interior operator   $\mathcal{I}: L^{X}\times L\to L^{X}$, the formula
 \[
 \mathcal T_{\mathcal I}(f)= \bigvee\{\alpha \in L\mid f\leqslant \mathcal{I}(f,\alpha)\},\,\,\,\ f\in L^{X},
 \]
 defines an $L$-fuzzy topology $\mathcal T_{\mathcal I}:L^{X}\rightarrow L$ on $X$.

\end{document}